\documentclass[12pt,a4paper]{amsart}

\usepackage{hyperref}
\usepackage{amsmath,amssymb,amsbsy,amsfonts,amsthm,latexsym,
amsopn,amstext,amsxtra,euscript,amscd}
\usepackage[all]{xy}
\usepackage{tikz}
\usetikzlibrary{shapes.geometric,positioning}

\newcommand{\FF}{\mathcal F}

\newtheorem{thm}{Theorem}
\newtheorem{cor}{Corollary}
\newtheorem{lem}[thm]{Lemma}

\begin{document}

\title{A note on product sets of rationals}

\author{Javier Cilleruelo}
\address{Instituto de Ciencias Matem\'aticas
 (CSIC-UAM-UC3M-UCM) and Departamento de Matem\'aticas, Universidad
 Aut\'onoma de Madrid, 28049 Madrid, Spain }



\begin{abstract}
Bourgain, Konyagin and Shparlinski obtained a lower bound for the size of the product set $AB$ when $A$ and $B$ are sets of positive rational numbers with numerator and denominator less or equal than $Q$. We extend and slightly improve that lower bound using a different approach. 
\end{abstract}


\maketitle


\section{Introduction}
Bourgain, Konyagin and Shparlinsky \cite{BKS} obtained a lower bound for the size of the product of two sets of rational numbers $$A,B\subset \FF_Q=\{q/q':\ 1\le q,q'\le Q\}$$ and they applied it to the study of the distribution of elements of multiplicative groups in residue rings. See  \cite{CRR}  and \cite{CG} for related results and more applications of this useful inequality.

\

\noindent \textbf{Theorem A} (BKSh).\label{BKS} {\it If $A,B\subset \mathcal F_{Q}$ then
\begin{equation}\label{AB}|AB|\ge |A||B|\exp\left (-(9+o(1))\log Q/\sqrt{\log \log Q}\right ),\end{equation}
where $o(1)\to 0$ when $Q\to \infty$.}

%
%
\

For any real numbers $Q,Q'\ge 1$ let
$\FF_{Q,Q'}$ denotes the set of rational numbers
$$\FF_{Q,Q'}=\{q/q':\ 1\le q\le Q,\ 1\le q'\le Q'\}.$$
We give the following result which extends and slightly improves Theorem A.
\begin{thm}\label{main} If $A,B\subset \mathcal F_{Q,Q'}$  then
$$|A/B|\ge |A||B|\exp\left (-(2\sqrt{\log 2}+o(1))\log(QQ')/\sqrt{\log \log(QQ')}\right ),$$
where $o(1)\to 0$ when $QQ'\to \infty$.
\end{thm}
Taking $Q'=Q$ and the set $1/B=\{b^{-1}:\ b\in B\}$ instead of $B$ we improve the constant in \eqref{AB}.
\begin{cor}\label{cor}If $A,B\in \mathcal F_{Q}$, then
$$|AB|\ge |A||B|\exp\left (-(4\sqrt{\log 2}+o(1) )\log Q/\sqrt{\log
\log Q}\right).$$
\end{cor}
\section{Proof of Theorem \ref{main}}
For any pair of sets $A,B\subset \FF_{Q,Q'}$ and $\gcd(r,s)=1$
we define the sets
\begin{eqnarray*}\mathcal M(A\times B,r/s)&=&\{(a/a',b/b')\in A\times B:\ \gcd(a,b)=r,\ \gcd(a',b')=s\}\\ A_{r/s}&=&\{a/a'\in A,\ r\mid a,\ s\mid
  a'\}\\ B_{r/s}&=&\{b/b'\in B,\ r\mid b,\ s\mid
  b'\}.\end{eqnarray*}

  It is clear that $\mathcal M(A\times B,r/s)\subset
  A_{r/s}\times B_{r/s}$, so we have
  \begin{equation}\label{times}
|\mathcal M(A\times B,r/s)|\le |A_{r/s}||B_{r/s}|.
  \end{equation}

  We claim that each $c/d\in A/B$ (assume that $\gcd(c,d)=1$) has at most $\tau(c)\tau(d)$ representation as \begin{equation}\label{cd}\frac cd=\frac{a/a'}
  {b/b'}\end{equation} with $(a/a',b/b')\in \mathcal M(A\times B,r/s)$.
Indeed we observe that \eqref{cd} implies
 $\frac cd=\frac{a_0b_0'} {b_0a_0'}$
  where $a_0=a/r,\quad b_0=b/r,\quad a_0'=a_0/s,\quad b_0'=b_0/s.$
    Since  $\gcd(c,d)=1$ and $\gcd(a_0b_0',a_0'b_0)=1$ then $c=a_0b_0'$ and $d=a_0'b_0$, which proves the claim.

  Note that
$c=a_0b_0'\le QQ'$ and $d=a_0'b_0\le QQ'$, thus the claim implies the inequality
\begin{equation}\label {eq1}|\mathcal M(A,B,r/s)|\le
T^2|A/B|,\end{equation} where $T=T(QQ')$ and $T(x)$ is the function \begin{equation*}\label{T}T(x)=\max_{m\le x}\tau(m).\end{equation*}

Using \eqref{times}, \eqref{eq1} and the well known inequality $$\sum_{\substack{1\le r,s\\rs\le x}}1\le x(1+\log x)$$ we get
\begin{eqnarray}\label{fromula}|A||B|&=&\sum_{\substack{rs\le x\\ (r,s)=1}}|\mathcal
M(A,B,r/s)|+\sum_{\substack{rs> x\\(r,s)=1}}|\mathcal M(A,B,r/s)|\\
&\le & T^2|A/B|x(1+\log
x)+\sum_{\substack{rs>x\\(r,s)=1}}|A_{r/s}||B_{r/s}|\nonumber \end{eqnarray}
for any real number $x\ge 1$.
If $x$ is such that the last sum is less than $|A||B|/2$ then we get
\begin{equation}\label{ine}|A/B|\ge \frac{|A||B|}{2T^2x(1+\log
x)}.\end{equation}
Now we are ready to prove the key Lemma.
\begin{lem}\label{n}For any $n\ge 1$ and for any $A,B\in
\FF_{Q,Q'}$ with real numbers $Q,Q'\ge 1$, we have \begin{equation}\label{induction}|A/B|\ge
\frac {|A||B|}{(4T)^{n+1}(QQ')^{1/n}(1+\log(QQ'))}\end{equation} where
$T=\max_{m\le QQ'}\tau(m).$
\end{lem}
\begin{proof}  We proceed by induction on $n$: trivially, since $|B|\le QQ'$  we have
$$|A/B|\ge |A|\ge\frac {|A||B|}{QQ'},$$
which proves \eqref{induction} for $n=1$. Suppose that Lemma \ref{n} is true for some $n\ge 1$.

If there is $r/s$ such that
\begin{eqnarray}\label{cota}|A_{r/s}||B_{r/s}|\ge \frac{(QQ')^{\frac
1{n(n+1)}}}{4T (rs)^{1/n}}|A||B|\end{eqnarray}  we use induction for the sets
$A_{r/s},B_{r/s}\subset \mathcal F_{Q/r,Q'/s}$. By observing  that the function $T(x)=\max_{m\le x}\tau(m)$ is a non decreasing function we have
\begin{eqnarray*}|A/B|&\ge &|A_{r/s}/B_{r/s}|\\ (\text{by induction hypothesis}) &\ge &
\frac {|A_{r/s}||B_{r/s}|}{(4T)^{n+1}((Q/r)(Q'/s))^{1/n}(1+\log((Q/r)(Q'/s)))}\\
( \text{by }\eqref{cota})\ &\ge &\frac{|A||B|}{(4T)^{n+2}(QQ')^{1/(n+1)}(1+\log(QQ'))}.\end{eqnarray*}

Thus, we assume that
$$|A_{r/s}||B_{r/s}|<
\frac{(QQ')^{\frac 1{n(n+1)}}}{4T(rs)^{1/n}}|A||B|$$ for any
$r/s,\ (r,s)=1$. In this case  we have
\begin{eqnarray}\label{rs}\sum_{rs>x}|A_{r/s}||B_{r/s}|&\le &\max_{rs>x}(|A_{r/s}||B_{r/s}|)^{1/2}\sum_{rs>x}|A_{r/s}|^{1/2}|B_{r/s}|^{1/2}\nonumber \\ &\le &  \frac{(QQ')^{\frac
1{2n(n+1)}}}{2T^{1/2}x^{\frac 1{2n}}}(|A||B|)^{1/2}\left
(\sum_{r,s}|A_{r/s}|\right )^{1/2}\left (\sum_{r,s}|B_{r/s}|\right
)^{1/2}.\end{eqnarray} To estimate the sums in the brackets we have
\begin{eqnarray}\label{t}\sum_{\substack{r,s}}|A_{r/s}|&=&\sum_{q/q'\in A}\sum_{\substack{r,s\\ r\mid q,\ s\mid q'}}1
\le  \sum_{q/q'\in A}\tau(qq')
\le |A|T.\end{eqnarray}
Putting in \eqref{rs} the estimate \eqref{t} and the analogous for $\sum_{r,s}|B_{r/s}|$ we have
$$\sum_{rs>x}|A_{r/s}||B_{r/s}|\le |A||B|\frac{T^{1/2}(QQ')^{\frac
1{2n(n+1)}}}{2x^{\frac 1{2n}}}.$$

Taking $x=T^n(QQ')^{\frac 1{n+1}}$ we get
$$\sum_{rs>x}|A_{r/s}||B_{r/s}|\le |A||B|/2.$$
Then \eqref{ine} applies and noting that $\log x\le \log ((QQ')^{n+\frac 1{n+1}})\le 2n\log (QQ')$ we get
\begin{eqnarray*}|A/B|&\ge &\frac{|A||B|}{2T^2x(1+\log x)}\\ &\ge &\frac{|A||B|}{2T^{n+2}(QQ')^{\frac
1{(n+1)}}(1+ 2n\log (QQ')  )}\\  &\ge &\frac{|A||B|}{(4T)^{n+2}(QQ')^{\frac
1{(n+1)}}(1+\log(QQ'))}\times \frac{2^{2n+3}(1+\log(QQ'))}{1+ 2n\log (QQ')   }\\&\ge&
\frac{|A||B|}{(4T)^{n+2}(QQ')^{\frac
1{(n+1)}}(1+\log(QQ'))}.\end{eqnarray*}
\end{proof}
The well known upper bound for the divisor function, $$\tau(m)\le \exp((\log 2+o(1))\log m/\log \log m)$$ 
implies $$T\le \exp((\log 2+o(1))\log (QQ')/\log\log (QQ')).$$
Thus, an optimal choice of $n$ in Lemma \ref{n} is $n\sim
\sqrt{\frac{\log \log (QQ')}{\log 2}}$, from where Theorem \ref{main} follows.

\subsection*{Acknowledgements}This work was supported by grants MTM 2011-22851 of MICINN and ICMAT Severo
Ochoa project SEV-2011-0087.


\end{document}